\documentclass{conm-p-l}
\usepackage{amssymb}
\usepackage{amsmath}
\usepackage{color}
\usepackage[matrix,arrow,curve,cmtip]{xy}
\entrymodifiers={+!!<0pt,\fontdimen22\textfont2>} 
\usepackage{units}
\usepackage{enumitem}
\setenumerate[1]{leftmargin=2em}

\theoremstyle{plain}   
\newtheorem{bigthm}{Theorem}   

\newtheorem{theorem}[equation]{Theorem}  
     
\newtheorem{lemma}[equation]{Lemma}         
\newtheorem{proposition}[equation]{Proposition} 

\theoremstyle{remark}
\newtheorem{remark}[equation]{Remark}


\begin{document}

\title{Algebraic $K$-theory of planar cuspidal curves}

\author{Lars Hesselholt}
\address{Nagoya University, Japan, and University of Copenhagen, Denmark}
\email{larsh@math.nagoya-u.ac.jp}

\author{Thomas Nikolaus}
\address{Universit\"{a}t M\"{u}nster, Germany}
\email{nikolaus@uni-muenster.de}


\maketitle

\section*{Introduction}

The purpose of this paper is to evaluate the algebraic $K$-groups of a
planar cuspidal curve over a perfect $\mathbb{F}_p$-algebra relative
to the cusp point. A conditional calculation of these 
groups was given in~\cite[Theorem~A]{h10}, assuming a conjecture on
the structure of certain polytopes. Our calculation here,
however, is unconditional and illustrates the advantage of the new
setup for topological cyclic homology by
Nikolaus--Scholze~\cite{nikolausscholze}, which we will be using
throughout. The only input necessary for our calculation is the
evaluation by the Buenos Aires Cyclic Homology group~\cite{bach1} and 
Larsen~\cite{larsen} of the structure of the Hochschild complex of the
coordinate ring as a mixed complex, that is, as an object of the
$\infty$-category of chain complexes with circle action.

We consider the planar cuspidal curve ``$y^a = x^b$,'' where
$a,b \geq 2$ are relatively prime integers. For $m \geq 0$, we define
$\ell(a,b,m)$ to be the number of pairs $(i,j)$ of positive integers
such that $ai+bj = m$, and for $r \geq 0$, define $S(a,b,r)$ to be
the set of positive integers $m$ such that $\ell(a,b,m) \leq r$. The
subset $S = S(a,b,r) \subset \mathbb{Z}_{>0}$ is a truncation set in
the sense that if $m \in S$ and $d$ divides $m$, then
$d \in S$, and hence, the ring of (big) Witt vectors $\mathbb{W}_S(k)$
with underlying set $k^S$ is defined. We refer to~\cite[Section~1]{h9}
for a detailed introduction to Witt vectors.

\begin{bigthm}\label{thm:cusps}Let $k$ be a perfect
$\mathbb{F}_p$-algebra, and let $a, b \geq 2$ be relatively prime
integers. There is a canonical isomorphism
$$K_j(k[x,y]/(y^a-x^b),(x,y)) \simeq
\mathbb{W}_S(k)/(V_a\mathbb{W}_{S/a}(k) + V_b\mathbb{W}_{S/b}(k)),$$
if $j = 2r \geq 0$ with $S = S(a,b,r)$, and the remaining $K$-groups
are zero.
\end{bigthm}

We remark that recently Angeltveit ~\cite{angeltveit} has given a
different proof of this result that, unlike ours, employs equivariant
homotopy theory. We also remark that the strategy employed in this
paper was used by Speirs~\cite{speirs} to significantly simplify the
calculation in~\cite{hm1} of the relative algebraic $K$-groups of a
truncated polynomial algebra over a perfect $\mathbb{F}_p$-algebra.

We recall from~\cite[Section~1]{h10} that the group in the
statement is a module over the ring $\mathbb{W}(k)$ of big Witt
vectors in $k$ of finite length $\frac{1}{2}(2r+1)(a-1)(b-1)$. The
calculation of the length goes back to
Sylvester~\cite{sylvester}. Moreover, it admits a $p$-typical product
decomposition indexed by positive integers $m'$ not divisible by
$p$. To state this, we let $s = s(a,b,r,p,m')$ be the unique integer 
such that
$$\ell(a,b,p^{s-1}m') \leq r < \ell(a,b,p^sm'),$$
if such an integer exists, and $0$, otherwise, and assume
(without loss of generality) that
$p$ does not divide $b$ and write $a = p^ua'$ with $a'$ not divisible
by $p$. Then
$$\textstyle{
\mathbb{W}_S(k)/(V_a\mathbb{W}_{S/a}(k) + V_b\mathbb{W}_{S/b}(k))
\simeq \prod_{m' \in \mathbb{N}'} W_h(k) }$$
by a canonical isomorphism, where
$$h = h(a,b,r,p,m') = \begin{cases}
s, & \text{if neither $a'$ nor $b$ divides $m'$,} \cr
\min\{s,u\}, & \text{if $a'$ but not $b$ divides $m'$,} \cr
0, & \text{if $b$ divides $m'$.} \cr
\end{cases}$$
Here we write $\mathbb{N}'$ for the set of positive integers not
divisible by $p$. 

It is a pleasure to acknowledge the generous support that we have
received while working on this paper. Hesselholt was funded in part by
the Isaac Newton Institute as a Rothschild Distinguished Visiting
Fellow and by the Mathematical Sciences Research Institute 
as a Simons Visiting Professor. Nikolaus was funded in part by the
Deutsche Forschungsgemeinschaft under Germany's Excellence Strategy
EXC 2044 Ð390685587, Mathematics M\"unster:
Dynamics--Geometry--Structure. Finally, we are grateful to Tyler
Lawson for pointing out that our arguments in an earlier version of
this paper could be simplified significantly. 

\section{Some recollections necessary for the
proof}\label{sec:recollections} 

We first recall the Nikolaus--Scholze formula for topological cyclic
homology from~\cite{nikolausscholze}; see
also~\cite{hesselholtnikolaus}. We write $\mathbb{T}$ for the circle
group and $C_p \subset \mathbb{T}$ for the subgroup of order $p$.
If $R$ is a ring, then we write
$$\begin{xy}
(0,0)*+{ \operatorname{TC}_{\phantom{-}}^{-}(R) }="1";
(25,0)*+{ \operatorname{TP}(R) }="2";
{ \ar^-{\operatorname{can}} "2";"1";};
\end{xy}$$
for the canonical map from the homotopy fixed points to the Tate
construction of the spectrum with $\mathbb{T}$-action
$\operatorname{THH}(R)$. The Frobenius map
$$\xymatrix{
{ \operatorname{THH}(R) } \ar[r]^-{\varphi} &
{ \operatorname{THH}(R)^{tC_p} } \cr
}$$
is $\mathbb{T}$-equivariant, provided that we let $\mathbb{T}$ act on
the target through the isomorphism
$\rho \colon \mathbb{T} \to \mathbb{T}/C_p$ given by the $p$th
root, and therefore, it induces a map
$$\begin{xy}
(0,0)*+{ \operatorname{TC}^{-}(R) = \operatorname{THH}(R)^{h\mathbb{T}} }="1";
(45,0)*+{ (\operatorname{THH}(R)^{tC_p})^{h(\mathbb{T}/C_p)}. }="2";
{ \ar^-{\varphi^{h\hskip.5pt\mathbb{T}}} "2";"1";};
\end{xy}$$
The Tate-orbit lemma~\cite[I.2.1, II.4.2]{nikolausscholze} identifies
the $p$-completion of the target of this map with that of
$\operatorname{TP}(R)$, and Nikolaus--Scholze show that, after
$p$-completion, the topological cyclic homology of $R$ is the
equalizer
$$\begin{xy}
(0,0)*+{ \operatorname{TC}(R) }="1";
(24,0)*+{ \operatorname{TC}^{-}(R) }="2";
(48,0)*+{ \operatorname{TP}(R) }="3";
{ \ar "2";"1";};
{ \ar@<.7ex>^-{\varphi} "3";"2";};
{ \ar@<-.7ex>_-{\operatorname{can}} "3";"2";};
\end{xy}$$
of these two parallel maps. If $p$ is nilpotent in $R$, as is the case
in the situation that we consider, then the spectra in question are
already $p$-complete.

The normalization of $A = k[x,y]/(y^a-x^b)$ is the $k$-algebra
homomorphism to $B = k[t]$ that to $x$ and $y$ assigns $t^a$ and
$t^b$, respectively, and this homomorphism identifies $A$ with
sub-$k$-algebra $k[t^a,t^b] \subset k[t] = B$. In this situation, the
square
$$\begin{xy}
(0,7)*+{ K(A) }="11";
(25,7)*+{ \operatorname{TC}(A) }="12";
(0,-7)*+{ K(B) }="21";
(25,-7)*+{ \operatorname{TC}(B) }="22";
{ \ar "12";"11";};
{ \ar "21";"11";};
{ \ar "22";"12";};
{ \ar "22";"21";};
\end{xy}$$
is cartesian. This follows from the birelative theorem, which has now
been given a very satisfying conceptual proof by
Land--Tamme~\cite{landtamme}. (By contrast, the original proof in the
rational case by Corti\~{n}as~\cite{cortinas} and the subsequent proof
in the $p$-adic case by Geisser--Hesselholt both required rather
elaborate calculational input.) Now, the map $K(B) \to K(k)$ induced
by the $k$-algebra homomorphism that to $t$ assigns $0$ is an
equivalence, and hence, the relative $K$-groups that we wish to
determine are canonically identified with the homotopy groups of the
common fibers of the vertical maps in the diagram above.

The $k$-algebras $A$ and $B$ are both monoid algebras. In general, if
$k[\Pi]$ is the monoid algebra of an $\mathbb{E}_1$-monoid $\Pi$ in
spaces, then, as cyclotomic spectra,
$$\operatorname{THH}(k[\Pi]) \simeq \operatorname{THH}(k \otimes
\mathbb{S}[\Pi]) \simeq \operatorname{THH}(k) \otimes
B^{\operatorname{cy}}(\Pi)_+,$$
where $B^{\operatorname{cy}}(\Pi)$ denotes the unstable cyclic
bar-construction of $\Pi$. In addition, on the right-hand side, the
Frobenius map factors as a composition
$$\begin{aligned}
{} & \xymatrix{
{ \operatorname{THH}(k) \otimes B^{\operatorname{cy}}(\Pi)_+ }
\ar[r]^-{\varphi \otimes \tilde{\varphi}} &
{ \operatorname{THH}(k)^{tC_p} \otimes
B^{\operatorname{cy}}(\Pi)_+^{hC_p} } \cr } \cr
{} & \xymatrix{
{ \phantom{ \operatorname{THH}(k) \otimes
B^{\operatorname{cy}}(\Pi)_+} } \ar[r]^-{\operatorname{can}} & 
{ (\operatorname{THH}(k) \otimes B^{\operatorname{cy}}(\Pi)_+)^{tC_p}
} \cr } \cr
\end{aligned}$$
of the map induced by the Frobenius
$\varphi \colon \operatorname{THH}(k) \to
\operatorname{THH}(k)^{tC_p}$ and the unstable Frobenius
$\tilde{\varphi} \colon B^{\operatorname{cy}}(\Pi) \to
B^{\operatorname{cy}}(\Pi)^{hC_p}$ and a canonical map. Importantly, we
have a map of spectra with  $\mathbb{T}$-action
$$\xymatrix{
{ \mathbb{Z} } \ar[r] &
{ \mathbb{Z}_p \simeq \tau_{\geq 0}\operatorname{TC}(k) } \ar[r] &
{ \operatorname{THH}(k) } \cr
}$$
from $\mathbb{Z}$ with (necessarily) trivial $\mathbb{T}$-action, and
therefore, we can rewrite
$$\operatorname{THH}(k) \otimes B^{\operatorname{cy}}(\Pi)_+
\simeq \operatorname{THH}(k) \otimes_{\mathbb{Z}} 
\mathbb{Z} \otimes B^{\operatorname{cy}}(\Pi)_+.$$
Accordingly, we do not need to understand the homotopy type of the
space with $\mathbb{T}$-action $B^{\operatorname{cy}}(\Pi)$. It
suffices to understand the homotopy type of the chain complex with
$\mathbb{T}$-action $\mathbb{Z} \otimes B^{\operatorname{cy}}(\Pi)_+$,
which, in the case at hand, is exactly what the Buenos Aires Cyclic
Homology group~\cite{bach1} and Larsen~\cite{larsen} have done for
us.

To state their result, we let
$\langle t^a,t^b \rangle \subset \langle t \rangle$ be the free
monoid on a generator $t$ and the submonoid generated by $t^a$ and
$t^b$, respectively, and set
$$B^{\operatorname{cy}}(\langle t \rangle, \langle t^a,t^b \rangle) =
B^{\operatorname{cy}}(\langle t \rangle) /
B^{\operatorname{cy}}(\langle t^a,t^b \rangle).$$
Counting powers of $t$ gives a $\mathbb{T}$-equivariant decomposition
of pointed spaces
$$\textstyle{ B^{\operatorname{cy}}(\langle t \rangle, \langle t^a,t^b
  \rangle) \simeq \bigvee_{m \in \mathbb{Z}_{>0}}
  B^{\operatorname{cy}}(\langle t \rangle, \langle t^a,t^b
  \rangle;m). }$$
We can now state the result of the calculation by the Buenos Aires
Cyclic Homology group~\cite{bach1} and by Larsen~\cite{larsen} as
follows. \goodbreak

\begin{theorem}\label{thm:buenosaires}Let $a,b \geq 2$ be
relatively prime integers, and let $m \geq 1$ be an integer. In the
$\infty$-category $D(\mathbb{Z})^{B\hspace{.5pt}\mathbb{T}}$ of chain complexes
with $\mathbb{T}$-action, there is a canonical equivalence between
$$\mathbb{Z} \otimes B^{\operatorname{cy}}(\langle t \rangle,\langle
t^a,t^b \rangle;m)$$
and the total cofiber of the square
$$\xymatrix{
{ \mathbb{Z} \otimes (\mathbb{T}/C_{m/ab})_+[2\ell(a,b,m)] } \ar[r]
\ar[d] &
{ \mathbb{Z} \otimes (\mathbb{T}/C_{m/a})_+[2\ell(a,b,m)] } \ar[d] \cr
{ \mathbb{Z} \otimes (\mathbb{T}/C_{m/b})_+[2\ell(a,b,m)] } \ar[r] &
{ \mathbb{Z} \otimes (\mathbb{T}/C_m)_+[2\ell(a,b,m)]. } \cr
}$$
Here, all maps in the square are induced by the respective canonical
projections, and if $c$ does not divide $m$, then
$\mathbb{T}/C_{m/c}$ is understood to be the empty space.
\end{theorem}

The result is not stated in this from in op.~cit., and therefore, some
explanation 
is in order. The $\infty$-category $D(\mathbb{Z})^{B\hspace{.5pt}\mathbb{T}}$ is
equivalent to the $\infty$-category of modules over the
$\mathbb{E}_1$-algebra $C_*(\mathbb{T},\mathbb{Z})$ given by the
singular chains on the circle. This $\mathbb{E}_1$-algebra, in turn,
is a Postnikov section of a free $\mathbb{E}_1$-algebra over
$\mathbb{Z}$, and therefore, it is formal in the sense that, as an
$\mathbb{E}_1$-algebra over $\mathbb{Z}$, it is equivalent to the
Pontryagin ring $H_*(\mathbb{T},\mathbb{Z})$ given by its homology. As
a model for the $\infty$-category of modules over the latter, we may
use the dg-category of dg-modules over
$H_*(\mathbb{T},\mathbb{Z})$. But a dg-module over
$H_*(\mathbb{T},\mathbb{Z}) = \mathbb{Z}[d]/(d^2)$ is precisely what
is called a mixed complex in op.~cit. This shows that
$D(\mathbb{Z})^{B\hspace{.5pt}\mathbb{T}}$ is equivalent to the $\infty$-category
of mixed complexes. Some additional translation is necessary to bring the
results in the above form for which we refer to~\cite[Section~5]{h10}.

In the following, we will use the abbreviation
$$B_m = B^{\operatorname{cy}}(\langle t \rangle, \langle t^a,t^b \rangle;m).$$ 
Theorem~\ref{thm:buenosaires} shows, in particular, that the
connectivity of $B_m$ tends to infinity with $m$. Hence, tensoring
with $\operatorname{THH}(k)$, we obtain an equivalence
$$\operatorname{THH}(k) \otimes B^{\operatorname{cy}}(\langle t
\rangle, \langle t^a,t^b \rangle) \simeq \textstyle{ \bigoplus_{m \in
    \mathbb{Z}_{>0}} \operatorname{THH}(k) \otimes B_m }
\simeq \textstyle{ \prod_{m \in \mathbb{Z}_{>0}} \operatorname{THH}(k)
  \otimes B_m, }$$
which, in turn, implies a product decomposition
$$\textstyle{ (\operatorname{THH}(k) \otimes
  B^{\operatorname{cy}}(\langle t \rangle, \langle t^a,t^b
  \rangle))^{tC_p} \simeq \prod_{m \in \mathbb{Z}_{>0}}
  (\operatorname{THH}(k) \otimes B_m)^{tC_p}. }$$
We will use the following result repeatedly below.

\begin{lemma}\label{lem:frobeniusmap}The unstable Frobenius induces an
equivalence  
$$\begin{xy}
(0,0)*+{ \operatorname{THH}(k)^{tC_p} \otimes B_{m/p} }="1";
(44,0)*+{ (\operatorname{THH}(k) \otimes B_m)^{tC_p}, }="2";
{ \ar^-(.48){\operatorname{id} \otimes \tilde{\varphi}} "2";"1";};
\end{xy}$$
where the left-hand side is understood to be zero if $p$ does not
divide $m$. 
\end{lemma}

\begin{proof}We recall two facts from~\cite[Section~3]{h10}. The first
is that the pointed space $B_{m/p}$ is finite, and the second is that
the cofiber of the composition
$$\xymatrix{
{ B_{m/p} } \ar[r]^-{\tilde{\varphi}} &
{ (B_m)^{hC_p} } \ar[r] &
{ B_m } \cr
}$$
of the unstable Frobenius map and the canonical ``inclusion'' is a finite
colimit of free pointed $C_p$-cells. Here $C_p$ acts trivially on the
left-hand and middle terms. Now, the map in the statement factors as
the composition
$$\xymatrix{
{ \operatorname{THH}(k)^{tC_p} \otimes B_{m/p} } \ar[r] &
{ (\operatorname{THH}(k) \otimes B_{m/p})^{tC_p} } \ar[r] &
{ (\operatorname{THH}(k) \otimes B_m)^{tC_p} } \cr
}$$
of the canonical colimit interchange map, where $B_{m/p}$ is equipped
with the trivial $C_p$-action, and the map of Tate spectra induced
from the composite map above. The first fact implies that the
left-hand map is an equivalence, and the second fact implies that the
right-hand map is an equivalence.
\end{proof}

\begin{proposition}\label{prop:duality}Let $G$ be a compact Lie group, let
$H \subset G$ be a closed subgroup, let $\lambda = T_H(G/H)$ be the
tangent space at $H = eH$ with the adjoint left $H$-action, and let
$S^{\lambda}$ the one-point compactification of $\lambda$. For every 
spectrum with $G$-action $X$, there are canonical natural equivalences
$$\begin{aligned}
(X \otimes (G/H)_+)^{hG} & \simeq (X \otimes S^{\lambda})^{hH}, \cr
(X \otimes (G/H)_+)^{tG} & \simeq (X \otimes S^{\lambda})^{tH}. \cr
\end{aligned}$$
\end{proposition}

\begin{proof}We recall that for every map of spaces
$f \colon S \to T$, the restriction functor
$f^* \colon \operatorname{Sp}^{T} \to \operatorname{Sp}^S$ has both a
left adjoint $f_{!}$ and a right adjoint $f_*$. In the case of the unique
map $p \colon BG \to \operatorname{pt}$, we have
$p_{\hskip.5pt!}(X) \simeq X_{hG}$ and $p_*(X) \simeq
X^{hG}$. We now consider the following diagram of spaces.
$$\begin{xy}
(0,7)*+{ BH }="11";
(26,7)*+{ BG }="13";
(13,-7)*+{ \operatorname{pt} }="22";
{ \ar^-{f} "13";"11";};
{ \ar_-(.35){q} "22";"11";};
{ \ar^-(.35){p} "22";"13";};
\end{xy}$$
The top horizontal map is the map induced by the inclusion of $H$ in
$G$. It is a fiber bundle, whose fibers are compact
manifolds.\footnote{\,Strictly speaking this statement does not make
sense, since $BH$ and $BG$ are only defined as homotopy types. What
we mean is that the map is classified by a map $BG \to
B\operatorname{Diff}(G/H)$, where the latter is the diffeomorphism
group of the compact manifold $G/H$.} Therefore, by parametrized
Atiyah duality, its relative 
dualizing spectrum $Df \in \operatorname{Sp}^{BH}$ is given by the
sphere bundle associated with the fiberwise normal bundle, which is
$Df \simeq S^{-\lambda}$. By definition of the dualizing spectrum, we
have for all $Y \in \operatorname{Sp}^{BH}$, a natural equivalence
$$f_{!}(Y \otimes S^{-\lambda} ) \simeq f_*(Y)$$
in $\operatorname{Sp}^{BG}$. It follows that for all
$Y \in \operatorname{Sp}^{BH}$, we have a natural equivalence
$$p_*f_{!}(Y \otimes S^{-\lambda}) \simeq p_*f_*(Y) \simeq q_*(Y)$$
in $\operatorname{Sp}$, which we also write as
$$((Y \otimes S^{-\lambda}) \otimes_H G_+)^{hG} \simeq Y^{hH}.$$
By~\cite[Theorem~I.4.1~(3)]{nikolausscholze}, we further deduce a
natural equivalence
$$((Y \otimes S^{-\lambda}) \otimes_H G_+)^{tG} \simeq Y^{tH}.$$
Indeed, the left-hand side vanishes for $Y = \Sigma^\infty H_+$,
and the fiber of the map
$$\begin{xy}
(0,0)*+{ ((Y \otimes S^{-\lambda})\otimes_H G_+)^{hG} }="1";
(47,0)*+{ ((Y \otimes S^{-\lambda})\otimes_H G_+)^{tG} }="2";
{ \ar^-{\operatorname{can}} "2";"1";};
\end{xy}$$
preserves colimits in $Y$.

Finally, given $X \in \operatorname{Sp}^{BG}$, we set 
$Y = f^*(X) \otimes S^{\lambda}$ to obtain the equivalences in
the statement.
\end{proof}

\section{Proof of Theorem~\ref{thm:cusps}}

We consider the diagram with horizontal equalizers
$$\begin{xy}
(0,7)*+{ \operatorname{TC}(A) }="11";
(24,7)*+{ \operatorname{TC}^{-}(A) }="12";
(48,7)*+{ \operatorname{TP}(A) }="13";
(0,-7)*+{ \operatorname{TC}(B) }="21";
(24,-7)*+{ \operatorname{TC}^{-}(B) }="22";
(48,-7)*+{ \operatorname{TP}(B) }="23";
{ \ar "12";"11";};
{ \ar@<.7ex>^-{\varphi} "13";"12";};
{ \ar@<-.7ex>_-{\operatorname{can}} "13";"12";};
{ \ar "22";"21";};
{ \ar "21";"11";};
{ \ar "22";"12";};
{ \ar "23";"13";};
{ \ar@<.7ex>^-{\varphi} "23";"22";};
{ \ar@<-.7ex>_-{\operatorname{can}} "23";"22";};
\end{xy}$$
and wish to evaluate the cofiber of the left-hand vertical map. We have
$$\textstyle{ \operatorname{cofiber}(\operatorname{THH}(A) \to
  \operatorname{THH}(B)) \simeq \bigoplus_{m \geq 1}
  \operatorname{THH}(k) \otimes B_m \simeq \prod_{m \geq 1}
  \operatorname{THH}(k) \otimes B_m. }$$
The right-hand equivalence follows from the fact that the connectivity of
$B_m$ tends to infinity with $m$. Therefore, the cofiber of
$\operatorname{TC}(A) \to \operatorname{TC}(B)$ is identified with the
equalizer of the induced maps
$$\xymatrix{
{ \prod_{m \geq 1} (\operatorname{THH}(k) \otimes B_m)^{h\mathbb{T}} }
\ar@<.7ex>[r]^{\varphi} \ar@<-.7ex>[r]_-{\operatorname{can}} &
{ \prod_{m \geq 1} (\operatorname{THH}(k) \otimes B_m)^{t\mathbb{T}}. } \cr
}$$
While the canonical map ``$\operatorname{can}$'' preserves this product
decomposition, the Frobenius map ``$\varphi$'' takes the factor indexed by
$m$ to the factor factor indexed $pm$. Therefore, we write $m = p^vm'$ with
$m'$ not divisible by $p$ and rewrite the diagram as
$$\xymatrix{
{ \prod_{m'} \prod_{v \geq 0} (\operatorname{THH}(k) \otimes
  B_{p^vm'})^{h\mathbb{T}} } \ar@<.7ex>[r]^{\varphi}
\ar@<-.7ex>[r]_-{\operatorname{can}} &
{ \prod_{m'} \prod_{v \geq 0} (\operatorname{THH}(k) \otimes
  B_m)^{t\mathbb{T}}, } \cr
}$$
where both maps now preserve the outer product decomposition indexed
by positive integers $m'$ not divisible by $p$. We will abbreviate
and write
$$\xymatrix{
{ \operatorname{TC}(m') } \ar[r] &
{ \operatorname{TC}^{-}(m') } \ar@<.7ex>[r]^-{\varphi}
\ar@<-.7ex>[r]_-{\operatorname{can}} &
{ \operatorname{TP}(m') } \cr
}$$
for the equalizer diagram given by the factors indexed by $m'$. To
complete the proof, we evaluate the induced diagram on homotopy
groups.

We fix $m = p^vm'$. It follows from Theorem~\ref{thm:buenosaires} that
$$\operatorname{THH}(k) \otimes B_m \simeq \operatorname{THH}(k)
\otimes_{\mathbb{Z}} (\mathbb{Z} \otimes B_m)$$
agrees, up to canonical equivalence, with the total cofiber of the
square\footnote{\,Here we use in an essential way that, 
as a spectrum with $\mathbb{T}$-action, $\operatorname{THH}(k)$ is a 
$\mathbb{Z}$-module.}
$$\xymatrix{
{ \operatorname{THH}(k) \otimes (\mathbb{T}/C_{m/ab})_+[2\ell(a,b,m)]
} \ar[r] \ar[d] &
{ \operatorname{THH}(k) \otimes (\mathbb{T}/C_{m/a})_+[2\ell(a,b,m)] }
\ar[d] \cr
{ \operatorname{THH}(k) \otimes (\mathbb{T}/C_{m/b})_+[2\ell(a,b,m)] }
\ar[r] &
{ \operatorname{THH}(k) \otimes (\mathbb{T}/C_m)_+[2\ell(a,b,m)]. }
\cr
}$$
By Proposition~\ref{prop:duality}, the induced square of
$\mathbb{T}$-homotopy fixed points takes the form
$$\xymatrix{
{ \operatorname{THH}(k)^{hC_{m/ab}}[2\ell(a,b,m)+1] } \ar[r] \ar[d] &
{ \operatorname{THH}(k)^{hC_{m/a}}[2\ell(a,b,m)+1] } \ar[d] \cr
{ \operatorname{THH}(k)^{hC_{m/b}}[2\ell(a,b,m)+1] } \ar[r] &
{ \operatorname{THH}(k)^{hC_m}[2\ell(a,b,m)+1] } \cr
}$$
with the maps in the diagram given by the  corestriction
maps on homotopy fixed points. Indeed, the adjoint representation 
$\lambda = T_{C_r}(\mathbb{T}/C_r)$ is a trivial one-dimensional real
$C_r$-representation. We now write $a = p^ua'$ with $a'$ not
divisible by $p$ and assume that $p$ does not divide $b$. If $a$ and
$b$ both do not divide $m$, then
$$\begin{aligned}
(\operatorname{THH}(k) \otimes B_m)^{h\mathbb{T}} &
\simeq \operatorname{THH}(k)^{hC_m}[2\ell(a,b,m)+1] \cr
{} & \simeq \operatorname{THH}(k)^{hC_{p^v}}[2\ell(a,b,m)+1], \cr
\end{aligned}$$
and if $a = p^ua'$ but not $b$ divides $m$, then
$$\begin{aligned}
(\operatorname{THH}(k) \otimes B_m)^{h\mathbb{T}} &
\simeq \operatorname{cofiber}( \operatorname{THH}(k)^{hC_{m/a}} \! \to
\operatorname{THH}(k)^{hC_m})[2\ell(a,b,m)+1] \cr
{} & \simeq \operatorname{cofiber}(
\operatorname{THH}(k)^{hC_{p^{v-u}}} \! \to
\operatorname{THH}(k)^{hC_{p^v}})[2\ell(a,b,m)+1]. \cr
\end{aligned}$$
Similarly, if $b$ but not $a$ divides $m$, then
$$\begin{aligned}
(\operatorname{THH}(k) \otimes B_m)^{h\mathbb{T}} &
\simeq \operatorname{cofiber}( \operatorname{THH}(k)^{hC_{m/b}} \! \to 
\operatorname{THH}(k)^{hC_m})[2\ell(a,b,m)+1] \cr
{} & \simeq \operatorname{cofiber}( \operatorname{THH}(k)^{hC_{p^v}}
\! \to \operatorname{THH}(k)^{hC_{p^v}})[2\ell(a,b,m)+1] \cr
{} & \simeq 0, \cr
\end{aligned}$$
and if $a$ and $b$ both divide $m$, then
$$(\operatorname{THH}(k) \otimes B_m)^{h\mathbb{T}} \simeq 0,$$
since $B_m \simeq 0$. By the same reasoning, we find that
$$(\operatorname{THH}(k) \otimes B_m)^{t\mathbb{T}} \simeq
\operatorname{THH}(k)^{tC_{p^v}}[2\ell(a,b,m)+1],$$
if $a$ and $b$ both do not divide $m$, that
$$(\operatorname{THH}(k) \otimes B_m)^{t\mathbb{T}} \simeq
\operatorname{cofiber}(\operatorname{THH}(k)^{tC_{p^{v-u}}} \!\!\to
\operatorname{THH}(k)^{tC_{p^v}})[2\ell(a,b,m)+1],$$
if $a = p^ua'$ divides $m$ but $b$ does not divide $m$, and that
$$(\operatorname{THH}(k) \otimes B_m)^{t\mathbb{T}} \simeq 0,$$
otherwise.

It follows from~\cite[Section~IV.4]{nikolausscholze} that, on homotopy
groups, the diagram
$$\begin{xy}
(0,7)*+{ \operatorname{THH}(k)^{hC_{p^u}} }="11";
(32,7)*+{ \operatorname{THH}(k)^{tC_{p^u}} }="12";
(0,-7)*+{ \operatorname{THH}(k)^{hC_{p^v}} }="21";
(32,-7)*+{ \operatorname{THH}(k)^{tC_{p^v}} }="22";
{ \ar^-{\operatorname{can}} "12";"11";};
{ \ar^-{\operatorname{can}} "22";"21";};
{ \ar^-{\operatorname{cor}} "21";"11";};
{ \ar^-{\operatorname{cor}} "22";"12";};
\end{xy}$$ 
becomes
$$\begin{xy}
(0,7)*+{ W(k)[t,x]/(tx-p,p^ut) }="11";
(50,7)*+{ W(k)[t^{\pm1},x]/(tx-p,p^ut)\phantom{,} }="12";
(0,-7)*+{ W(k)[t,x]/(tx-p,p^vt) }="21";
(50,-7)*+{ W(k)[t^{\pm1},x]/(tx-p,p^vt), }="22";
{ \ar "12";"11";};
{ \ar "22";"21";};
{ \ar "21";"11";};
{ \ar "22";"12";};
\end{xy}$$
where $\deg(t) = -2$ and $\deg(x) = 2$, where the horizontal maps
are the unique graded $W(k)$-algebra homomorphisms that take $t$ to
$t$ and $x$ to $x = pt^{-1}$, and where the vertical maps are the
unique maps of graded $W(k)[t,x]$-modules that take $1$ to $p^{v-u}$.

We have now determined the diagram
$$\xymatrix{
{ \operatorname{TC}(m') } \ar[r] &
{ \operatorname{TC}^{-}(m') } \ar@<.7ex>[r]^-{\varphi}
\ar@<-.7ex>[r]_-{\operatorname{can}} &
{ \operatorname{TP}(m') } \cr
}$$
at the level of homotopy groups, the Frobenius given by
Lemma~\ref{lem:frobeniusmap}. Hence, it is merely a matter of
bookkeeping to see that the statement of the theorem ensues. We recall
the functions $s = s(a,b,r,p,m')$ and $h = h(a,b,r,p,m')$ from the
$p$-typical decomposition recalled in the introduction,
$$\textstyle{
\mathbb{W}_S(k)/(V_a\mathbb{W}_{S/a}(k) +
V_b\mathbb{W}_{S/b}(k)) \simeq \prod_{m' \in \mathbb{N}'} W_h(k). }$$
Suppose first that neither $a'$ nor $b$ divides $m'$. Then
$$\begin{aligned}
\pi_{2r+1}((\operatorname{THH}(k) \otimes
B_{p^vm'})^{h\hskip.5pt\mathbb{T}}) {} & \simeq \begin{cases}
W_{v+1}(k), & \text{if $0 \leq v < s$,} \cr
W_v(k), & \text{if $s \leq v$,} \cr
\end{cases} \cr
\pi_{2r+1}((\operatorname{THH}(k) \otimes
B_{p^vm'})^{t\hskip.5pt\mathbb{T}}) 
{} & \simeq W_v(k), \cr
\end{aligned}$$
with $s = s(a,b,r,p,m')$. Also, we remark that the corresponding
homotopy groups in even degree $2r$ are zero. The Frobenius map
$$\begin{xy}
(0,0)*+{ \pi_{2r+1}((\operatorname{THH}(k) \otimes
  B_{p^vm'})^{h\hskip.5pt\mathbb{T}}) }="1";
(56,0)*+{ \pi_{2r+1}((\operatorname{THH}(k) \otimes
  B_{p^{v+1}m'})^{t\hskip.5pt\mathbb{T}}) 
}="2";
{ \ar^-{\varphi} "2";"1";};
\end{xy}$$
is an isomorphism for $0 \leq v < s$, and the canonical map
$$\begin{xy}
(0,0)*+{ \pi_{2r+1}((\operatorname{THH}(k) \otimes
  B_{p^vm'})^{h\hskip.5pt\mathbb{T}}) }="1"; 
(55,0)*+{ \pi_{2r+1}((\operatorname{THH}(k) \otimes
  B_{p^vm'})^{t\hskip.5pt\mathbb{T}}) 
}="2";
{ \ar^-{\operatorname{can}} "2";"1";};
\end{xy}$$
is an isomorphism for $s \leq v$. Hence, we have a map of exact
sequences
$$\xymatrix{
{ 0 } \ar[r] &
{ \prod_{s \leq v} W_v(k) } \ar[r]
\ar[d]^-{\varphi-\operatorname{can}} &
{ \operatorname{TC}_{2r+1}^{-}(m') } \ar[r]
\ar[d]^-{\varphi-\operatorname{can}} &
{ \prod_{0 \leq v < s} W_{v+1}(k) } \ar[r]
\ar[d]^-(.43){\overline{\varphi-\operatorname{can}}} &
{ 0\phantom{,} } \cr
{ 0 } \ar[r] &
{ \prod_{s \leq v} W_v(k) } \ar[r] &
{ \operatorname{TP}_{2r+1}(m') } \ar[r] &
{ \prod_{0 \leq v < s} W_v(k) } \ar[r] &
{ 0, } \cr
}$$
where the left-hand vertical map is an isomorphism, and where the
right-hand vertical map is an epimorphism with kernel $W_s(k)$. Since
$h = s$, we conclude that $\operatorname{TC}_{2r+1}(m') \simeq W_h(k)$
and $\operatorname{TC}_{2r}(m') \simeq 0$ as desired.

Suppose next that $a'$ but not $b$ divides $m'$. If $u \leq s$, then
$$\begin{aligned}
\pi_{2r+1}((\operatorname{THH}(k) \otimes
B_{p^vm'})^{h\hskip.5pt\mathbb{T}})
{} & \simeq \begin{cases}
W_{v+1}(k), & \text{if $0 \leq v < u$,} \cr
W_u(k), & \text{if $u \leq v$,} \cr
\end{cases} \cr
\pi_{2r+1}((\operatorname{THH}(k) \otimes
B_{p^vm'})^{t\hskip.5pt\mathbb{T}})
{} & \simeq \begin{cases}
W_v(k), & \text{if $0 \leq v < u$,} \cr
W_u(k), & \text{if $u \leq v.$ } \cr
\end{cases}
\end{aligned}$$
Moreover, we see as before that the Frobenius and canonical maps are
isomorphisms for $0 \leq v < s$ and $s \leq v$, respectively, so we
have a map of exact sequences
$$\xymatrix{
{ 0 } \ar[r] &
{ \prod_{s \leq v} W_u(k) } \ar[r]
\ar[d]^-{\varphi-\operatorname{can}} &
{ \operatorname{TC}_{2r+1}^{-}(m') } \ar[r]
\ar[d]^-{\varphi-\operatorname{can}} &
{ \prod_{0 \leq v < s} W_{c}(k) } \ar[r]
\ar[d]^-(.43){\overline{\varphi-\operatorname{can}}} &
{ 0\phantom{,} } \cr
{ 0 } \ar[r] &
{ \prod_{s \leq v} W_u(k) } \ar[r] &
{ \operatorname{TP}_{2r+1}(m') } \ar[r] &
{ \prod_{0 \leq v < s} W_d(k) } \ar[r] &
{ 0, } \cr
}$$
where $c = \min\{u,v+1\}$ and $d = \min\{u,v\}$. The left-hand
vertical map is an isomorphism, and the right-hand vertical map is an
epimorphism with kernel $W_u(k)$, so again
$\operatorname{TC}_{2r+1}(m') \simeq W_h(k)$, since $h = u$, and
$\operatorname{TC}_{2r}(m') \simeq 0$.

If $s < u$, then
$$\begin{aligned}
\pi_{2r+1}((\operatorname{THH}(k) \otimes
B_{p^vm'})^{h\hskip.5pt\mathbb{T}})
{} & \simeq \begin{cases}
W_{v+1}(k), & \text{if $0 \leq v < s$,} \cr
W_v(k), & \text{if $s \leq v < u$,} \cr
W_u(k), & \text{if $u \leq v$,} \cr
\end{cases} \cr
\pi_{2r+1}((\operatorname{THH}(k) \otimes
B_{p^vm'})^{t\hskip.5pt\mathbb{T}})
{} & \simeq \begin{cases}
W_v(k), & \text{if $0 \leq v < u$,} \cr
W_u(k), & \text{if $u \leq v,$ } \cr
\end{cases}
\end{aligned}$$
and a similar argument shows that
$\operatorname{TC}_{2r+1}(m') \simeq W_h(k)$, since $h = s$, and that
$\operatorname{TC}_{2r}(m') \simeq 0$.

Finally, if $b$ divides $m'$, then $\operatorname{TC}_{2r+1}^{-}(m')$ and
$\operatorname{TP}_{2r+1}(m')$ are both zero, and therefore, so is
$\operatorname{TC}_{2r+1}(m')$ and $\operatorname{TC}_{2r}(m')$. This
completes the proof.

\begin{remark}As a pointed space with $\mathbb{T}$-action, the
homotopy type of $B_m$ was described conjecturally
in~\cite[Conjecture~B]{h10} and this conjecture was affirmed by
Angeltveit in~\cite[Theorem~2.1]{angeltveit}. The result is that $B_m$
is equivalent to the total cofiber of a square of pointed spaces with
$\mathbb{T}$-action
$$\xymatrix{
{ (\mathbb{T}/C_{m/ab})_+ \wedge S^{\lambda(a,b,m)} } \ar[r]
\ar[d] &
{ (\mathbb{T}/C_{m/a})_+ \wedge S^{\lambda(a,b,m)} } \ar[d] \cr
{ (\mathbb{T}/C_{m/b})_+ \wedge S^{\lambda(a,b,m)} } \ar[r] &
{ (\mathbb{T}/C_m)_+ \wedge S^{\lambda(a,b,m)}. } \cr
}$$
So the trivial shift $[2\ell(a,b,m)]$ that appears in
Theorem~\ref{thm:buenosaires} is replaced by a shift by a non-trivial
representation $\lambda(a,b,m)$ of real dimension $2\ell(a,b,m)$. 
\end{remark}

\bibliographystyle{amsplain}

\begin{thebibliography}{10}

\bibitem{angeltveit}
V.~Angeltveit.
\newblock Picard groups and the {$K$}-theory of curves with cuspidal
  singularities.
\newblock arXiv:1901.00264.

\bibitem{cortinas}
G.~Corti\~nas.
\newblock The obstruction to excision in {$K$}-theory and in cyclic homology.
\newblock {\em Invent. Math.}, 164:143--173, 2006.

\bibitem{bach1}
J.~A. Guccione, J.~J. Guccione, M.~J. Redondo, and O.~E. Villamayor.
\newblock Hochschild and cyclic homology of hypersurfaces.
\newblock {\em Adv. Math.}, 95:18--60, 1992.

\bibitem{h10}
L.~Hesselholt.
\newblock On the {$K$}-theory of planar cuspical curves and a new family of
  polytopes.
\newblock In {\em Algebraic Topology: Applications and New Directions}, volume
  620 of {\em Contemp. Math.}, pages 145--182, Stanford, CA, July 23-27, 2012,
  2014. Amer. Math. Soc., Providence, RI.

\bibitem{h9}
L.~Hesselholt.
\newblock The big de~{R}ham--{W}itt complex.
\newblock {\em Acta Math.}, 214:135--207, 2015.

\bibitem{hm1}
L.~Hesselholt and I.~Madsen.
\newblock Cyclic polytopes and the {$K$}-theory of truncated polynomial
  algebras.
\newblock {\em Invent. Math.}, 130:73--97, 1997.

\bibitem{hesselholtnikolaus}
L.~Hesselholt and T.~Nikolaus.
\newblock Topological cyclic homology.
\newblock Handbook of Homotopy Theory, Chap.~15 (to appear).

\bibitem{landtamme}
M.~Land and G.~Tamme.
\newblock On the {$K$}-theory of pullbacks.
\newblock Ann. of Math. (to appear).

\bibitem{larsen}
M.~Larsen.
\newblock Filtrations, mixed complexes, and cyclic homology in mixed
  characteristic.
\newblock {\em $K$-Theory}, 9:173--198, 1995.

\bibitem{nikolausscholze}
T.~Nikolaus and P.~Scholze.
\newblock On topological cyclic homology.
\newblock {\em Acta Math.}, 221:203--409, 2018.

\bibitem{speirs}
M.~Speirs.
\newblock On the {$K$}-theory of truncated polynomial algebras, revisited.
\newblock arXiv:1901.10602.

\bibitem{sylvester}
J.~J. Sylvester.
\newblock On subvariants, i.e. semi-invariants to binary quantics of an
  unlimited order.
\newblock {\em Amer. J. Math.}, 5:79--136, 1882.

\end{thebibliography}

\end{document}